\documentclass{amsart}
\usepackage{graphicx}
\usepackage{cite}
\usepackage{amssymb}
\newtheorem{thm}{Theorem}[section]
\newtheorem{cor}[thm]{Corollary}
\newtheorem{lem}[thm]{Lemma}

\theoremstyle{definition}
\newtheorem{defn}[thm]{Definition}
\theoremstyle{remark}
\newtheorem{rem}[thm]{Remark}
\numberwithin{equation}{section}

\begin{document}

\title[An inhomogeneous polyharmonic Dirichlet problem with $L^{p}$ data]{An inhomogeneous polyharmonic Dirichlet problem with $L^p$ boundary data in the upper half-plane}%
\author{Kanda Pan, Guoan Guo and Zhihua Du}%
\address{Department of Mathematics, Jinan University, Guangzhou 510632, China}%
\email{243642500@qq.com (P. Kan)}%
\address{Colledge of Mathematics, Nanjing University of Posts and Telecommunications, Nanjing 210023, China}%
\email{guoguoan@njupt.edu.cn(G. Guo)}%
\address{Department of Mathematics, Jinan University, Guangzhou 510632, China}%
\email{tzhdu@jnu.edu.cn (Z. Du)}%

\thanks{This work was carried out when the third named author visited Department of Mathematics, Temple University by the invitation from Prof. Irina Mitrea and the support from State Scholarship Fund Award of China. He was also partially supported by the NNSF grants (Nos. 11126065, 11401254). The second named author was partially supported by SRF of NJUPT(\#NY208070).}%
\subjclass{31B10, 31B30}%
\keywords{Dirichlet problem; polyharmonic equation; higher order Poisson kernels; higher order Pompeiu operators; inhomogeneous, non-tangential maximal functions}%

\dedicatory{Dedicate to Professor Dr. Heinrich Begehr\\ for his lasting contributions to the theory of integral representations}%
\begin{abstract}
  In this paper, it is investigated for an inhomogeneous Dirichlet problem with $L^p$ boundary data for polyharmonic equation in the upper half-plane. By using higher order Poisson kernels and Pompeiu operators, which are respectively due to Du, Qian and Wang [Z. Du, T. Qian and J. Wang, {\it $L^{p}$ polyharmonic Dirichlet problems in regular domains II: The upper half plane}, J. Differential Equations 252(2012), 1789-1812] as well as Begehr and Hile [H. Begehr and G. Hile, {\it A hierarchy of integral operators}, Rocky Mountain J. Math. 27(1997), 669-706], it is given that the unique integral representation solution under some certain estimates.
\end{abstract}
\maketitle
\section{Introduction}

In recent years, a great deal of activities were given to investigate boundary value problems (simply, BVPs) for higher order elliptic partial differential equations in various planar and higher dimensional domains. There were a lot of results exhibiting the development in this field \cite{bdw,bg,bh,bkpv,dqw1,dqw2,dqw3,dgw,dkw,du1,du2,imm,pv,pv1,pv2,pv3,pv4,v1,v2,v3}. Generally speaking, the development included two directions: one was to find out the explicit solutions for the model equations (such as biharmoic, polyharmonic, polyanalytic equations etc.) on some regular domains (for instance, the unit disc or ball, the upper half-plan or -space and so on); the other was to study the existence and uniqueness of solutions for general elliptic equations on general (non-smooth) domains (such as $C^{1}$, Lipschitz, non-tangentially accessible domains (see \cite{jk1}) and so on) under some different a priori estimates. Now in both directions many works are mainly concentrated on the study of BVPs with low regularity data on the coefficients of the equations and the boundary of the domains. With such view, this article is touching on a result in the former direction. More precisely, this article is devoted to solve the following inhomogeneous polyharmonic Dirichlet problems with $L^p$ boundary data in the upper half-plane, $\textbf{H}$, i.e.
\begin{equation}
\begin{cases}
\Delta^{n}u=g \ \ in \ \textbf{H},\\
\Delta^{j}u=f_{j} \ on \ \mathbb{R}\\
\end{cases}
\end{equation}
with $\|\textrm{M}(u-\widetilde{T}_{n,n,\mathbf{H}}g-\sum_{j=1}^{n-1}M_{j+1}[f_{j}-\widetilde{T}_{n-j, n-j, \mathbf{H}}g])\|_{L^{p}(\mathbb{R})}\leq C(\|f_{0}\|_{L^p(\mathbb{R})}+\|g\|_{L^{p}_{w}(\mathbf{H})})$, where $0\leq j\leq n-1$, $\textbf{H}=\{z\in \mathbb{C}: {\rm Re} z\in \mathbb{R}\,\, \mathrm{and}\,\, {\rm Im} z>0\}$, $\partial \textbf{H}=\mathbb{R}$, $\Delta(=\Delta_{z})=4\partial_{z}\partial_{\bar{z}}$ is the Laplacian with $\partial_{z}=\frac{1}{2}(\frac{\partial}{\partial_{x}}-i\frac{\partial}{\partial_{y}})$ and $\partial_{\bar{z}}=\frac{1}{2}(\frac{\partial}{\partial_{x}}+i\frac{\partial}{\partial_{y}})$, $z=x+iy\in \mathbf{H}$, and $\mathbb{R}$ is the real axis; $f_{j}\in L^p(\mathbb{R})$, $g\in L_{w}^{p}(\mathbf{H})$ which is a weighted $L^{p}$ space with certain weight $w$, $p>1$; $\widetilde{T}_{n-j,n-j,\mathbf{H}}$ and $M_{j}$ are respectively higher order Pompeiu operators  and Poisson integrals defined in Section 2 and 3;  $\mathrm{M}(u)$ is the non-tangential maximal function of $u$, which is usually defined by $$\textrm{M}(F)(x)=\sup\limits_{z\in \Gamma_{\alpha}(x)} |F(z)|\ \ \mathrm{for} \  x\in \partial \mathbb{R},$$ where $\Gamma_{\alpha}(x)$ is the non-tangential approach cone with the vertex at $(x, 0)$ and the aperture $\alpha>0$, viz.,
$$\Gamma_{\alpha}(x)=\{z\in \textbf{H}: |\mathrm{Re}z-x|<\alpha \mathrm{Im}z\}.$$
It is noteworthy that all the boundary data in BVP
(1.1) are non-tangential.

In 2008, Begehr, Du and Wang studied the same boundary value problem with H\"older continuous data on the unit disc but for the homogeneous polyharmonic equation \cite{bdw}. In their paper, Begehr, Du and Wang found that an integral representation solution for the problem could be given by some kernel functions satisfying some certain properties. Although there existed an inductive relation by the Laplacian among the kernel functions, it was hard to get the explicit expressions for the kernel functions and the integral representation solution by their method in terms of iterated poly-Cauchy integral operators. By studying the properties of the kernel functions stated in \cite{dgw} and introducing some new ideas, Du, Guo and Wang firstly gave the unified explicit expressions for all the kernels functions on the unit disc in terms of some convergent series (see \cite{dgw,du1} for details, however, more concisely and understandably appeared in \cite{dkw}). From then on, such kernel functions were called higher order Poisson kernels since they are higher order analogues of the classical Poisson kernel. Furtherly, Due to Du et al., the higher order Poisson kernels were explicitly constructed for the upper half-plane, the unit ball and the upper half-space respectively, and the corresponding $L^{p}$ homogeneous  polyharmonic Dirichlet problems (simply, PHD problems) on these domains were surely resolved by giving the integral representation solutions in terms of higher order Poisson kernels as kernel densities of the integrals \cite{dqw1,dqw2,dqw3}. More earlier, in 1997, to study higher order complex elliptic BVPs, Begehr and Hile introduced a class of kernel functions and defined a hierarchy of integral operators in terms of these kernels (see \cite{bh}),  which are higher order analogues of the classical Pompeiu operators, or $T$ and $\overline{T}$ operators as well as $\Pi$ and $\overline{\Pi}$ operators familiarly analyzed in Vekua's theory of generalized analytic functions \cite{vek}. So we call these integral operators introduced by Begehr and Hile to be higher order Pompeiu operators. In fact, Begehr and Hile's kernels and integral operators are extremely useful in the study of explicit solutions for higher order complex elliptic BVPs but their importance were neglected for a long time. Even by using higher order Poisson kernels and Begehr-Hile kernels, it will be easy to get some Green functions associated with some  higher order complex elliptic operators for some planar domains. In this paper, we will take advantage of higher order Poisson kernels for the upper half-plane and Begehr-Hile integral operators (or higher order Pompeiu operators) to get the integral representation solution of the BVP (1.1) under a certain estimate, and as a byproduct to obtain a Green function associated with the polyharmonic operators on the upper half-plane. Nevertheless, the results in this paper are only the tip of the iceberg as applications of Begehr-Hile kernels and operators. Under more smooth conditions for the boundary data, the corresponding results in the case of the unit disc can be found in the dissertation due to Du \cite{du1}.

\section{ Higher order Poisson kernels and Pompeiu opertors}

In this section, we sketchily present some results about the higher order Poisson kernels and Pompeiu operators. The details for these results can be seen in \cite{dqw1,bh}.
\subsection {Higher order Poisson kernels}
It is well known that the classical Poisson kernel is the key ingredient in giving the explicit solution (i.e., Poisson integral) to the Dirichlet problem for Laplace equation on the upper half-plane (see \cite{gar,st}). To solve a corresponding problem for the homogeneous polyharmonic equation on the upper half-plane, in \cite{dqw1}, Du, Qian and Wang constructed the higher order Poisson kernels for the upper half-plane (There these kernels were called to be higher order Schwarz kernels since they were expressed in terms of complex variables), which are higher order analogues of the classic Poisson kernel for the upper half-plane. The precise definition of these kernels is as follows:

\begin{defn} A sequence of real-valued functions of two variables $\{G_{n}(\cdot,\,\cdot)\}^{\infty}_{n=1}$ defined on $\textbf{H}\times \mathbb{R}$ is called  a sequence of higher order Poisson kernels, and, precisely, $G_{n}(\cdot,\,\cdot)$ is the $nth$ order Poisson kernel, if they satisfy the following conditions:
\begin{itemize}
\item [(1)]  For all $n\in \mathbb{N}, G_{n}(\cdot,\,\cdot)\in C(\textbf{H}\times \mathbb{R})$; $G_{n}(\cdot,\, t)\in C^{2n}(\textbf{H})$ with any fixed $t\in \mathbb{R}$; and $G_{n}(z,\cdot)\in L^{p}(\mathbb{R}), \, p>1$, with any fixed
$z\in \textbf{H}$, and the non-tangential boundary value
$$\lim_{z\rightarrow s \atop z\in \textbf{H},s\in \mathbb{R}} G_{n}(z,t)=G_{n}(s,t)$$
exists for all $t$ and $s\neq t$; $G_{n}(\cdot,\,t)$ can be continuously extended to $\overline{\textbf{H}} \backslash \{t\}$ for any fixed $t\in \mathbb{R}$;\\
\item [(2)] $G_{1}(z,t)=\frac{1}{2i}(\frac{1}{t-z}-\frac{1}{t-\bar{z}})$ and $G_{n}(i,t)=0, n\geq2$ and $t\in \mathbb{R}$, and for any $n\in \mathbb{N}$
\begin{equation*}|G_{n}(z,t)| \leq \frac{M}{|t-z'|}\end{equation*}
uniformly on $D_{c}\times\{t\in \mathbb{R}:|t|>T\}$ whenever $z'\in D_{c}$, where $D_{c}$ is any compact set in $\overline{\textbf{H}}$, $M$, $T$ are positive constants depending only on $D_{c}$ and $n$;\\
\item [(3)] $(\partial_{z}\partial_{\bar{z}})G_{1}(z,t)=0$ and $\Delta G_{n}(z,t)=G_{n-1}(z,t)$ for $n>1$;\\
\item [(4)] $\lim\limits_{z\rightarrow s,z\in \mathbf{H}}\frac{1}{\pi}\int^{+\infty}_{-\infty}G_{1}(z,t)\gamma(t)dt=\gamma(s)$, a.e., for any $\gamma \in L^{p}(\mathbb{R}),\,p\geq1$;\\
\item [(5)] $\lim\limits_{z\rightarrow s,z\in \mathbf{H}}\frac{1}{\pi}\int^{+\infty}_{-\infty}G_{n}(z,t)\gamma(t)dt=0$, for any $\gamma \in L^{p}(\mathbb{R}), \,p\geq1,\,n\geq2$,
\end{itemize}
where all the above limits are non-tangential.
\end{defn}

To get the explicit expressions of these kernels, the following decomposition theorem of polyharmonic functions is crucial.
\begin{lem}[\cite{dgw,du1}]
Let D be a simply connected (bounded or unbounded) domain in the complex plane with smooth boundary $\partial D$, $f$ be a real-valued $n$-harmonic function defined on $D$ (i.e., $\Delta^{n}f=0$ in $D$), then for any $z_0\in D$, there exist functions $f_{j}$, which is analytic in $D$ and has at least $j$th order zero at $z_{0}$, $j=0,1,\dots,n-1$ such that $$f(z)=2{\rm Re}\left\{\sum\limits_{j=0}^{n-1}(\bar{z}-\bar{z_0})^{j}f_{j}(z)\right\},z\in D,$$
where $\rm {Re}$ denotes the real part. Moreover, the above decomposition of $f$ is unique in some certain sense (more precisely, see \cite{dgw}).
\end{lem}

\begin{rem}
The above lemma is elementary. The biharmonic case was due to Goursat \cite{gor}, while two variants of the polyharmonic case can be respectively seen in \cite{ba,bdw}.
\end{rem}

In our case, let $D=\mathbf{H}$ and $z_{0}=i$, by essentially using the above lemma, we can obtain

\begin{lem} [\cite{dqw1}]
If $G_{n}(z,t)_{n=1}^{\infty}$ is a sequence of higher order Poisson kernels defined on $\bf{H}\times \mathbb{R}$, i.e., $G_{n}(z,t)_{n=1}^{\infty}$ fulfills the aforementioned properties (1)-(5) in Definition 2.1, then, for $n>1$, there exist functions $G_{n,0}(z,t), G_{n,1}(z,t),\dots,G_{n,n-1}(z,t)$ defined on $\mathbf{H}\times \mathbb{R}$ such that
$$G_{n}(z,t)=2{\rm Re}\left\{\sum\limits^{n-1}_{j=0}(\bar{z}+i)^{j}G_{n,j}(z,t)\right\},z\in {\bf H},t\in \mathbb{R},$$
with\\
$$\partial_z G_{n,j}(z,t)=j^{-1}G_{n-1,j-1}(z,t)$$
for $1\leq j\leq n-1$, \\
$$\partial^{k}_{z}G_{n,j}(i,t)=0$$
for $1\leq k\leq j-1$ with respect to $t\in \mathbb{R}$ and
$$G_{n,0}(z,t)=-\sum\limits^{n-1}_{j=1}(z+i)^{j}G_{n,j}(z,t).$$
Moreover, $$G_{1}(z,t)=\frac{1}{2i}\left(\frac{1}{t-z}-\frac{1}{t-\bar{z}}\right)$$
is the classical Poisson kernel for the upper half plane. All of the above $G_{n,j}\in (H\times L^p)({\bf H}\times \mathbb{R})$ (i.e., $G_{n,j}$ is continuous in ${\bf H}\times \mathbb{R}$, $G_{n,j}(\cdot,\,t)$ is analytic in ${\bf H}$ for any fixed $t\in \mathbb{R}$ and $G_{n,j}(z,\,\cdot)\in L^{p}(\mathbb{R})$ for any fixed $z\in {\bf H}$) , the non-tangential boundary value
$$\lim_{z\rightarrow s \atop z\in {\mathrm H},s\in \mathbb{R}}G_{n,j}(z,t)=G_{n,j}(s,t)$$
exists on $\mathbb{R}$, except $t\in \mathbb{R}$ and $G_{n,j}(s,\,\cdot)\in L^p(\mathbb{R})$ for any fixed $s\in \mathbb{R}$. We can further show that $G_{n,j}(\cdot,\,t)$ can be continuously extended to
$\overline{{\bf H}} \backslash \{t\}$ for any fixed $t\in \mathbb{R}$, and
$$|G_{n,j}(z,t)|\leq M\frac{1}{|t-z'|}$$
uniformly on $D_{c}\times\{t\in \mathbb{R}:|t|>T\}$ whenever $z'\in D_{c}$ which is any compact set in $\overline{{\bf H}}$, where $M$, $T$ are positive constants depending only on $D_{c}$.\\
Moreover,
$$\lim_{z\rightarrow s \atop z\in \overline{{\bf H}},S\in \mathbb{R}}|G_{n,j}(z,s)|=+\infty \ and \ \lim_{z\rightarrow s \atop z\in \overline{{\bf H}},S\in \mathbb{R}}|(z-s)G_{n,j}(z,s)|=0$$
for any $s\in \mathbb{R}$ and $n\geq2$.
\end{lem}

In fact, Lemma 2.4 has provided an algorithm to obtain all explicit expressions of higher order Poisson kernels. The explicit formulae are in the following

\begin{lem}[\cite{dqw1}]
Let $G_{n,n-1}$ and $G_{n}$ be stated as in Lemma 2.4, then for any
$n\geq2$,
\begin{eqnarray*}
G_{n,n-1}(z,t)&=&\frac{(n-3)!}{(n-1)!}(z-t)G_{n-1,n-2}(z,t)\\
              &&+\frac{1}{(n-1)!\times(n-2)!\times(n-2)\times2i}(z-i)^{n-2}\nonumber\\
              &=&\frac{1}{(n-1)!\times(n-2)!\times 2i}(z-t)^{n-2}\log\frac{t-i}{t-z}\nonumber\\
              &&+\sum_{j=1}^{n-2}\frac{1}{(n-1)!\times(n-2)!\times j\times 2i}(z-t)^{n-2-j}(z-i)^{j}
\end{eqnarray*}
and
\begin{eqnarray}
G_{n}(z,t)&=&2{\rm Re}\Big\{(\overline{z}-z)\Big[\frac{1}{(n-1)!\times(n-2)!\times 2i}|z-t|^{2(n-2)}\log\frac{t-i}{t-z}\\
              &&+\sum_{j=1}^{n-2}\frac{1}{(n-1)!\times(n-2)!\times j\times 2i}(\overline{z}-t)^{n-2}(z-t)^{n-2-j}\nonumber\\
              &&\times(z-i)^{j}\nonumber\\
              &&+\sum_{j=n-1}^{2(n-2)}\sum_{l=0}^{n-2}\frac{1}{(n-1)!\times(n-2)!\times j\times 2i}C_{n-2}^{l}(\overline{z}-z)^{l}\nonumber\\
              &&\times(z-t)^{2(n-2)-l-j}(z-i)^{j}\Big]\Big\}\nonumber
\end{eqnarray}
\begin{eqnarray*}
              &=&2{\rm Re}\Big\{(\overline{z}-z)\Big[\frac{1}{(n-1)!\times(n-2)!\times 2i}|z-t|^{2(n-2)}\log\frac{t-i}{t-z}\nonumber\\
              &&+\sum_{j=1}^{n-2}\frac{1}{(n-1)!\times(n-2)!\times j\times 2i}(\overline{z}-t)^{n-2}(z-t)^{n-2-j}\nonumber\\
              &&\times(z-i)^{j}\nonumber\\
              &&+\sum_{j=n-1}^{2(n-2)}\sum_{l=j}^{2(n-2)}\frac{1}{(n-1)!\times(n-2)!\times j\times 2i}C_{n-2}^{l-j}(\overline{z}-z)^{l-j}\nonumber\\
              &&\times(z-t)^{2(n-2)-l}(z-i)^{j}\Big]\Big\}\nonumber
\end{eqnarray*}
where $C_{n-2}^{j}$ are the binomial coefficients, $z\in \mathbf{H}$ and $t\in \mathbb{R}$.
\end{lem}

\subsection{Higher order Pompeiu operators} In \cite{bh}, Begehr and Hile firstly introduced a class of kernels and systematically defined a hierarchy of integral operators (Although some special class of these integral operators have already appeared in the works due to Dzhuraev \cite{dz}). Their integral operators are some extensions of the area integral appearing in the classical Cauchy-Pompeiu formula. The latter is a weakly singular integral operator which was called $T$ operator in the Vekua's theory of generalized analytic function, which and its adjoint $\overline{T}$ operator, as well as $\Pi$ operator and its adjoint $\overline{\Pi}$ operator (formally, $\Pi=:\partial_{z}T$ in the sense of classical or weak derivatives), play an important role in the study of Beltrami and generalized Beltrami equations as well as some second order complex elliptic equations.

More precisely, the Begehr-Hile kernels are defined as follow

\begin{defn}[\cite{bh}] Let $m$ and $n$ be integers, with $m+n\geq0$ but $(m,n)\neq(0,0)$, define
\begin{equation}K_{m,n}(z)=\begin{cases}
\frac{(-m)!(-1)^{m}}{(n-1)!\pi}z^{m-1}\bar{z}^{n-1},\, \,m\leq0;\\\\
\frac{(-n)!(-1)^{n}}{(m-1)!\pi}z^{m-1}\bar{z}^{n-1},\, \, \,n\leq0;\\\\
\frac{1}{(m-1)!(n-1)!\pi}z^{m-1}\bar{z}^{n-1}\left[\log|z|^{2}-\sum\limits_{k=1}^{m-1}\frac{1}{k}-\sum\limits^{n-1}_{l=1}\frac{1}{l}\right],\, \, m,n\geq1,
\end{cases}
\end{equation}
where the summations are zero when $m=1$ or $n=1$.
\end{defn}

By means of the above kernels, Begehr and Hile introduced the following convolution integral operators, which are higher order analogues of $T$ and $\Pi$ operators. So we call them higher order Pompeiu operators.
\begin{defn}[\cite{bh}]
For $D$ a domain (bounded or unbounded) in the plane, formally define operators $T_{m,n,D}$, acting on suitable complex valued functions \textit{w} defined in $D$, according to
\begin{align}T_{m,n,D}\textit{w}(z)&=\int\int_{D}K_{m,n}(z-\zeta)\textit{w}(\zeta)d\xi d\eta\\
&=\int\int_{\mathbb{C}}K_{m,n}(z-\zeta)\textit{w}(\zeta)\chi_{D}(\zeta)d\xi d\eta\nonumber\\
&=K_{m,n}*(w\chi_{D})(z),\nonumber\end{align}
where $\chi_{D}$ is the characteristic function of $D$, $z=x+iy$ and $\zeta=\xi+i\eta$.
\end{defn}

Obviously,
\begin{equation*}
T_{0,1,D}w=-\frac{1}{\pi}\int\int_{D}\frac{w(\zeta)}{\zeta-z}d\xi d\eta,\,\,\,T_{1,0,D}w=-\frac{1}{\pi}\int\int_{D}\frac{w(\zeta)}{\overline{\zeta}-\overline{z}}d\xi d\eta,
\end{equation*}
and
\begin{equation*}
T_{-1,1,D}w=-\frac{1}{\pi}\int\int_{D}\frac{w(\zeta)}{(\zeta-z)^{2}}d\xi d\eta,\,\,\,T_{1,-1,D}w=-\frac{1}{\pi}\int\int_{D}\frac{w(\zeta)}{(\overline{\zeta}-\overline{z})^{2}}d\xi d\eta,
\end{equation*}
so they are respectively just the operators $T_{D}w$, $\overline{T}_{D}w$, $\Pi_{D}w$ and $\overline{\Pi}_{D}w$ in Vekua's theory of generalized analytic functions. As a convention, $T_{0,0,D}=I$ which is the identity operator.

When $D$ and $w$ satisfy some regularity properties, $T_{m,n, D}w$ possesses some nice properties such as H\"older continuity, $L^{p}$ integrability and differentiability. For instance, let $D=\mathbb{C}$ and $w$ be a complex measurable function satisfying that
\begin{equation}|\textit{w}(z)|=o(|z|^{-m-n-\delta}),\ \ \ \ as \ z\rightarrow\infty\end{equation} for some $\delta>0$, then as $m+n\geq 2$ and $w\in L^{1}_{loc}(\mathbb{C})$, or $m+n\geq 1$ and $w\in L^{p}_{loc}(\mathbb{C})$ with $p>1$,
$$\partial_{z}T_{m,n, \mathbb{C}}\textit{w}=T_{m-1,n, \mathbb{C}}\textit{w}\,\,\,\,{\rm and}\,\,\, \,\partial_{\bar{z}}T_{m,n, \mathbb{C}}\textit{w}=T_{m,n-1, \mathbb{C}}\textit{w}$$
hold wholly in $\mathbb{C}$ in the sense of Sobolev's generalized derivatives. Moreover,
$$\partial_{z}T_{1,0, \mathbb{C}}\textit{w}=\partial_{\bar{z}}T_{0,1, \mathbb{C}}\textit{w}=\textit{w}$$
also hold in $\mathbb{C}$ even as $p=1$ in the sense of Sobolev's generalized derivatives.

In the sequel, we only need to use the kernels $K_{n,n}$ and the operators $T_{n,n, D}$ with $n\geq 1$ in a modified version. In our case, for any $p\geq 1$, it is easy to see that $K_{n,n}(z-\zeta)$ is not $L^{p}$ integrable with respect to $z\in \mathbf{H}$ and $\zeta\in \mathbb{R}$ when the another variable is fixed, and $T_{n,n, \mathbf{H}}w(z)$ is not $L^{p}$ integrable in $\mathbf{H}$, $\mathbb{R}$ or $\overline{\mathbf{H}}$ even if $w$ obeys the condition (2.4). So we must make some modifications for $K_{n,n}$ and $T_{n,n, \mathbf{H}}$ to obtain the $L^{p}$ integrability for $T_{n,n, \mathbf{H}}w$, which will be a crucial ingredient in the method of below. Before tackling the main problem, we need some preliminaries stated in what follows.

\begin{defn}
Let $f(z)$ be a continuous function defined in $\mathbb{C}$ which can be expanded as
\begin{equation*}
f(z)=\sum_{n=-\infty}^{m}c_{n}(z)|z|^{n}
\end{equation*}
for sufficiently large $|z|$, where $m\geq -2$ and the coefficient functions $c_{n}(z)$ are continuous in $\mathbb{C}$. Denote
\begin{equation*}
\mathrm{S. P.} [f](z)=\sum_{n=-2}^{m}c_{n}(z)|z|^{n}
\end{equation*}
and
\begin{equation*}
\mathrm{I. P.} [f](z)=\sum_{n=3}^{\infty}c_{-n}(z)|z|^{-n}
\end{equation*}
for sufficiently large $|z|$. If $\mathrm{I. P.} [f](z)$ is $L^{p}$ integrable in $\{z\in \mathbb{C}: |z|\geq R\}$ with some sufficiently large $R>0$, then $\mathrm{S. P.} [f]$ is said to be the singular part of $f$ and $\mathrm{I. P.} [f]$ is said to be the integrable part of $f$ in the sense of $L^{p}$ integrable, $p\geq 1$.
\end{defn}

To make the kernels $K_{n,n}$ become $L^{p}$ integrable, according to the above definition, we need try to expand them as a sum of singular part and integral part of them (then discard their singular parts). Fortunately, this can be done by introducing the ultraspherical (or say, Gegenbauer) polynomials, $P^{(\lambda)}_{l}$ and $Q_{l}^{(\lambda)}$, which
can be defined respectively by the generating functions as follows:

\begin{equation}
(1-2r\xi+r^{2})^{-\lambda}=\sum_{l=0}^{\infty}P_{l}^{(\lambda)}(\xi)r^{l}
\end{equation}
and
\begin{equation}
(1-2r\xi+r^{2})^{-\lambda}\log(1-2r\xi+r^{2})=\sum_{l=0}^{\infty}Q_{l}^{(\lambda)}(\xi)r^{l},
\end{equation}
where $\lambda\neq0$, $0\leq|r|<1$ and $|\xi|\leq1$.
$P_{l}^{(\lambda)}$ and $Q_{l}^{(\lambda)}$ have the following
explicit expressions:
\begin{align}
P_l^{(\lambda)}(\xi)=&\frac{1}{l!}\left\{\frac{d^{l}}{dr^{l}}\left[(1-2r\xi+r^{2})^{-\lambda}\right]\right\}_{r=0}\\
=&\sum_{j=0}^{[\frac{l}{2}]}(-1)^j\frac{\Gamma(l-j+\lambda)}{\Gamma(\lambda)j!(l-2j)!}(2\xi)^{l-2j}\nonumber
\end{align}
and
\begin{align}
Q_l^{(\lambda)}(\xi)=&-\frac{d}{d \lambda}\left[P_l^{(\lambda)}(\xi)\right]\\
=&\sum_{j=0}^{[\frac{l}{2}]}\sum_{k=0}^{l-j-1}(-1)^{j+1}\frac{\Gamma(l-j+\lambda)}{(\lambda+k)\Gamma(\lambda)j!(l-2j)!}(2\xi)^{l-2j},\nonumber
\end{align}
where $[\frac{l}{2}]$ denotes the integer part of $\frac{l}{2}$. For some special values of $\lambda$, say
$\lambda=\lambda_{0}$, the above expressions may be extended and
interpreted as limits for $\lambda\rightarrow\lambda_{0}$ (for
example, $\lambda$ is a non-positive integer). The properties
of the ultraspherical polynomials can be also found in
\cite{aar,sz}.

For sufficiently large $|\zeta|\geq |z|$ and any real numbers
$\lambda\neq0$,
\begin{align}
|z-\zeta|^{-2\lambda}&=(|\zeta|^{2}-2\mathrm{Re}(z
\overline{\zeta})+|z|^{2})^{-\lambda}\\
&=|\zeta|^{-2\lambda}\left[1-2\left(\frac{|z|}{|\zeta|}\right)\cos(\theta-\vartheta)+\left(\frac{|z|}{|\zeta|}\right)^{2}\right]^{-\lambda}\nonumber\\
&=|\zeta|^{-2\lambda}\sum_{l=0}^{\infty}P_{l}^{(\lambda)}(\cos(\theta-\vartheta))\left(\frac{|z|}{|\zeta|}\right)^{l}\nonumber\\
&=\sum_{l=0}^{\infty}|z|^{l}P_{l}^{(\lambda)}(\cos(\theta-\vartheta))|\zeta|^{-(l+2\lambda)},\nonumber
\end{align}
where $z=|z|e^{i\theta}$ and $\zeta=|\zeta|e^{i\vartheta}$.

Similarly, we have
\begin{align}
&|z-\zeta|^{-2\lambda}\log|z-\zeta|^{2}\\
=&|z-\zeta|^{-2\lambda}\left[\log\frac{|z-\zeta|^{2}}{|\zeta|^{2}}+2\log|\zeta|\right]\nonumber\\
=&(|\zeta|^{2}-2\mathrm{Re}(z
\overline{\zeta})+|z|^{2})^{-\lambda}\left[\log\frac{|\zeta|^{2}-2\mathrm{Re}(z
\overline{\zeta})+|z|^{2}}{|\zeta|^{2}}+2\log|\zeta|\right]\nonumber\\
=&|\zeta|^{-2\lambda}\left[1-2\left(\frac{|z|}{|\zeta|}\right)\cos(\theta-\vartheta)+\left(\frac{|z|}{|\zeta|}\right)^{2}\right]^{-\lambda}\Big\{\log\left[1-2\left(\frac{|z|}{|\zeta|}\right)\cos(\theta-\vartheta)+\left(\frac{|z|}{|\zeta|}\right)^{2}\right]\nonumber\\
&+2\log|\zeta|\Big\}\nonumber\\
=&|\zeta|^{-2\lambda}\sum_{l=0}^{\infty}Q_{l}^{(\lambda)}(\cos(\theta-\vartheta))\left(\frac{|z|}{|\zeta|}\right)^{l}+2|\zeta|^{-2\lambda}\log|\zeta|\sum_{l=0}^{\infty}P_{l}^{(\lambda)}(\cos(\theta-\vartheta))\left(\frac{|z|}{|\zeta|}\right)^{l}\nonumber\\
=&\sum_{l=0}^{\infty}|z|^{l}Q_{l}^{(\lambda)}(\cos(\theta-\vartheta))|\zeta|^{-(l+2\lambda)}+2\sum_{l=0}^{\infty}|z|^{l}\log|\zeta|P_{l}^{(\lambda)}(\cos(\theta-\vartheta))|\zeta|^{-(l+2\lambda)}.\nonumber
\end{align}

With the above preliminaries, we define modified Begehr-Hile kernels and Higher order Pompeiu operators on the upper half-plane in the case of $m=n$ as follows

\begin{defn}
Let $K_{n,n}$ be as above, and $w$ be a suitable complex function in $\mathbf{H}$, then for any $z, \zeta\in \mathbf{H}$ and $z\neq\zeta$, define
\begin{equation}
\widetilde{K}_{n,n}(z, \zeta)=\begin{cases}
K_{n,n}(z-\zeta),\,\,\,|z|=|\zeta|,\vspace{2mm}\\
K_{n,n}(z-\zeta)-\mathrm{S. P.} [K_{n,n}](z,\zeta),\,\,\,|z|\neq|\zeta|,
\end{cases}
\end{equation}
where
\begin{align}
\mathrm{S. P.} [K_{n,n}](z,\zeta)=&\frac{1}{\pi[(n-1)!]^{2}}
\times\Big\{\Big[\sum_{l=0}^{2n}Q_{l}^{(1-n)}(\cos(\theta-\vartheta))\\
&\times\min\left(\left|\frac{z+i}{\zeta+i}\right|^{l}, \left|\frac{z+i}{\zeta+i}\right|^{-l}\right)\times\max\left(|z+i|^{2n-2}, |\zeta+i|^{2n-2}\right)\Big]\nonumber\\
&+2\left[\log(\max(|z+i|, |\zeta+i|))-\sum_{k=1}^{n-1}\frac{1}{k}\right]\nonumber\\
&\times\Big[\sum_{l=0}^{2n}P_{l}^{(1-n)}(\cos(\theta-\vartheta))\times\min\left(\left|\frac{z+i}{\zeta+i}\right|^{l}, \left|\frac{z+i}{\zeta+i}\right|^{-l}\right)\nonumber\\
&\times\max\left(|z+i|^{2n-2}, |\zeta+i|^{2n-2}\right)\Big]\Big\}\nonumber
\end{align}
and
\begin{equation}
\widetilde{T}_{n, n, \mathbf{H}}w(z)=\int\int_{\mathbf{H}}\widetilde{K}_{n,n}(z, \zeta)w(\zeta)d\xi d\eta,
\end{equation}
where $z+i=|z+i|e^{i\theta}$ and $\zeta+i=|\zeta+i|e^{i\vartheta}$, $\widetilde{K}_{n,n}$ and $\widetilde{T}_{n, n, \mathbf{H}}$ are respectively called to be modified Begehr-Hile kernels and higher order Pompeiu operators of $(n,n)$-typed on the upper half-plane, more concisely, $(n,n)$-typed mBH kernels and mHOP operators on $\mathbf{H}$. (In Section 4, we also simply use $\widetilde{T}_{n, \mathbf{H}}$ instead of $\widetilde{T}_{n, n,\mathbf{H}}$.)
\end{defn}

\begin{rem}In the above definition, the $-i$ can be replaced by any other point $z_{0}\in \mathbb{C}\setminus \overline{\mathbf{H}}$, i.e., any point in the lower-half plane, whose role is to assure the integrability of $\mathrm{S. P.} [K_{n,n}]$ in any bounded subset of $\overline{\mathbf{H}}$.
\end{rem}

Now we study some properties of the kernels $\widetilde{K}_{n,n}$ and the operators $\widetilde{T}_{n,n, \mathbf{H}}$ for any $n\in \mathbb{N}$.

\begin{thm}[Estimate of $\widetilde{K}_{n,n}$] Let $D$ be a compact subset of $\overline{\mathbf{H}}$, denote $M_{D}=\max_{\zeta\in D}|z+i|$, then
\begin{equation}
|\widetilde{K}_{n,n}(z, \zeta)|\leq\frac{M}{|\zeta+i|^{2+\epsilon}}
\end{equation}
for any $z\in D$ and $\zeta\in \overline{\mathbf{H}}\setminus B(-i, R)$ with $R>M_{D}$, where $0<\epsilon<1$, the constant $M$ depends only on $n, \epsilon,  M_{D}$ and $R$.
\end{thm}

\begin{proof}
Since $z\in D$ and $\zeta\in \overline{\mathbf{H}}\setminus B(-i, R)$ with $R>M_{D}$, then $|z+i|<|\zeta+i|$, by a similar argument as the following (2.40), we further have that
\begin{align}
|\widetilde{K}_{n,n}(z, \zeta)|=|\mathrm{I. P.}[K_{n,n}](z, \zeta)|&\leq C\left(n, M_{D}/R\right)|z+i|^{2n+1}\left[1+\big|\log|\zeta+i|\big|\right]\frac{1}{|\zeta+i|^{3}}\\
&\leq C\left(n, \epsilon, M_{D}/R\right)M_{D}^{2n+1}\frac{1}{|\zeta+i|^{2+\epsilon}}\nonumber
\end{align}
for any $0<\epsilon<1$, where $C(\cdots)$ are some constants depending only on the quantities in the parentheses. In the last inequality, we have used the elementary fact $\lim_{\rho\rightarrow +\infty}\frac{\log\rho}{\rho^{1-\epsilon}}=0$ as $0<\epsilon<1$.

\end{proof}

\begin{thm}
Let $\widetilde{K}_{n,n}$ be as in the above definition, then for any $z,\zeta\in \mathbf{H}$ and $z\neq \zeta$,
\begin{itemize}
  \item [(1)] \,\,$\widetilde{K}_{n,n}(z,\zeta)=\widetilde{K}_{n,n}(\zeta, z)$;
  \item [(2)] \,\,$\Delta_{z}\widetilde{K}_{n,n}(z,\zeta)=4\widetilde{K}_{n-1,n-1}(z, \zeta)$, where $\Delta_{z}=4\frac{\partial^{2}}{\partial z\partial \overline{z}}=\frac{\partial^{2}}{\partial x^{2}}+\frac{\partial^{2}}{\partial y^{2}}$ and $z=x+iy$.
\end{itemize}
\end{thm}

\begin{proof}
By the definitions of $K_{n,n}$ and $\widetilde{K}_{n,n}$ , the claim (1) is obvious. To verify (2), by a direct calculation, we have
\begin{equation}
\Delta_{z}K_{n,n}(z,\zeta)=4K_{n-1,n-1}(z, \zeta)
\end{equation}
for any $z, \zeta\in \mathbb{C}$ and $z\neq\zeta$. So
\begin{equation}
\Delta_{z}[\widetilde{K}_{n,n}(z,\zeta)+\mathrm{S. P.} [K_{n,n}](z, \zeta)]=4\widetilde{K}_{n-1,n-1}(z, \zeta)+4\mathrm{S. P.} [K_{n-1,n-1}](z, \zeta),
\end{equation}
further
\begin{equation}
\Delta_{z}\widetilde{K}_{n,n}(z,\zeta)-4\widetilde{K}_{n-1,n-1}(z, \zeta)=4\mathrm{S. P.} [K_{n-1,n-1}](z, \zeta)-\Delta_{z}\mathrm{S. P.} [K_{n,n}](z, \zeta)
\end{equation}
for any $z, \zeta\in \mathbf{H}$ and $z\neq\zeta$. For sufficiently large $|\zeta|$ (thus $|\zeta+i|$), by Definition 2.8, Theorem 2.11 and the expansion expression of $\widetilde{K}_{n,n}$ as a series of $\frac{|z+i|}{|\zeta+i|}$, the RHS of (2.18)=$O(1/|\zeta+i|^{2+\epsilon})$ with $0<\epsilon<1$, while the LHS of (2.18)=$O(1/|\zeta+i|^{2})$+ higher order terms of $|\zeta+i|$. Since the series is absolutely convergent when $|z+i|\ll |\zeta+i|$, then by differentiating term by term and comparing the coefficients (with respect to $\frac{1}{|\zeta+i|^{l}}$) of two hand sides of (2.18), we obtain that
\begin{equation}
\Delta_{z}\widetilde{K}_{n,n}(z,\zeta)=4\widetilde{K}_{n-1,n-1}(z, \zeta)\end{equation}
and
\begin{equation}
\Delta_{z}\mathrm{S. P.} [K_{n,n}](z, \zeta)=4\mathrm{S. P.} [K_{n-1,n-1}](z, \zeta)
\end{equation}
hold for any sufficiently large $|\zeta|$ (and thus for all $|\zeta|$).
\end{proof}

\begin{rem}
Note that
\begin{equation}
\partial_{z}K_{n,n}(z,\zeta)=K_{n-1,n}(z, \zeta)\,\,\,and \,\,\,\partial_{\overline{z}}K_{n,n}(z,\zeta)=K_{n,n-1}(z, \zeta),
\end{equation}
by a similar argument, we can prove that
\begin{equation}
\partial_{z}\widetilde{K}_{n,n}(z,\zeta)=\widetilde{K}_{n-1,n}(z, \zeta)\,\,\,and \,\,\,\partial_{\overline{z}}\widetilde{K}_{n,n}(z,\zeta)=\widetilde{K}_{n,n-1}(z, \zeta)\end{equation}
and
\begin{equation}
\partial_{z}\mathrm{S. P.} [K_{n,n}](z, \zeta)=\mathrm{S. P.} [K_{n-1,n}](z, \zeta)\,\,\,and \,\,\,\partial_{\overline{z}}\mathrm{S. P.} [K_{n,n}](z, \zeta)=\mathrm{S. P.} [K_{n,n-1}](z, \zeta).
\end{equation}
By (2.20) and (2.23), there are some identities implied in them for which we will not pursue to detailedly get in the present paper.
\end{rem}

\begin{defn}
Let $w$ be a nonnegative locally integrable function defined on $\mathbf{H}$ with values in $(0, \infty)$ almost everywhere, For any $k,\alpha>0$, if $w$ satisfies that
\begin{itemize}
  \item [(i)] $|\zeta+i|^{k+\alpha}\left(1+\big|\log|\zeta+i|\big|\right)w^{-1}(\zeta)\in L^{\infty}(\mathbf{H})$;
  \item [(ii)] $|\zeta+i|^{k}\left(1+\big|\log|\zeta+i|\big|\right)^{p}|\zeta+i|^{\alpha}w^{-1}(\zeta)\in L^{\frac{1}{p-1}}(\mathbf{H})$ as $p\geq 1$,
\end{itemize}
then $w$ is called to be a $(p, k, \alpha)$-weight on $\mathbf{H}$. Denote $W^{p,k,\alpha}(\mathbf{H})$ be the set that consisting of all $(p,k,\alpha)$-weights on $\mathbf{H}$. (Note that, two conditions (i) and (ii) are the same as $p=1$.)
\end{defn}

\begin{rem}[The properties of $W^{p,k,\alpha}(\mathbf{H})$] It is easy to check the following nest relations:
\begin{equation}
W^{p,k,\alpha}(\mathbf{H})\subset W^{p,l,\alpha}(\mathbf{H})\subset W^{q,l,\alpha}(\mathbf{H})\subset W^{q,l,\beta}(\mathbf{H})
\end{equation}
for any $p>q\geq 1$, $k>l>0$ and $\alpha>\beta>0$.
\end{rem}

\begin{thm}[$L^{p}$ boundedness of $\widetilde{T}_{n,n, \mathbf{H}}$]
Let $\widetilde{T}_{n, n, \mathbf{H}}$ be as in Definition 2.9, $w_{1}\in W^{p,n,\frac{3}{2}}(\mathbf{H})$ and $w_{2}\in W^{p,n,\frac{5}{2}}(\mathbf{H})$ then $\widetilde{T}_{n, n, \mathbf{H}}$ is bounded as a linear operator from $L_{w_{1}}^{p}(\mathbf{H})$ to $L^{p}(\mathbf{\mathbb{R}})$, or from $L_{w_{2}}^{p}(\mathbf{H} )$ to $L^{p}(\mathbf{\mathbf{H}})$ for any $p\geq 1$. More precisely,
\begin{equation}
\|\widetilde{T}_{n, n, \mathbf{H}}f\|_{L^{p}(\mathbb{R})}\leq C\|f\|_{L_{w_{1}}^{p}(\mathbf{H})}
\end{equation}
and
\begin{equation}
\|\widetilde{T}_{n, n, \mathbf{H}}f\|_{L^{p}(\mathbf{H})}\leq C\|f\|_{L_{w_{2}}^{p}(\mathbf{H})}
\end{equation}
for any $p\geq 1$, where $L_{w_{j}}^{p}(\mathbf{H})$ denotes the set of all weighted $L^{p}$ integrable functions on $\mathbf{H}$ with the weight $w_{j}$, $j=1,2$; $C$ are some constants depending only on $n$ and $p$.
\end{thm}

\begin{rem}[Subspaces of $L^{p}(\mathbf{H})$] In fact, from the above condition (i) in Definition 2.14, $L^{p}_{w}(\mathbf{H})$ becomes a subspace of $L^{p}(\mathbf{H})$ (by a standard argument, i.e., using the completeness of $L^{p}(\mathbf{H})$ and Fatou's lemma). Here $L^{p}_{w}(\mathbf{H})$ denotes the set of all weighted $L^{p}$ integrable functions on $\mathbf{H}$ with the weight $w\in W^{p,k,\alpha}(\mathbf{H})$.
\end{rem}

\begin{proof}
We only verify (2.25), (2.26) is similar. To do so, we first establish the following estimate
\begin{equation}
\int_{\mathbb{R}}|\widetilde{K}_{n,n}(x, \zeta)|^{p}dx\leq C_{n, p}[|\zeta+i|^{2n-2}(1+\big|\log|\zeta+i|\big|)]^{p}|\zeta+i|^{\frac{3}{2}},
\end{equation}
where $C_{n, p}$ is a constant depending only on $n$ and $p$.

To get the above estimate, for any fixed $\zeta\in \mathbf{H}$, split
\begin{align}
\int_{\mathbb{R}}|\widetilde{K}_{n,n}(x, \zeta)|^{p}dx=\left(\int_{|x+i|\leq 2|\zeta+i|,\,\, x\in \mathbb{R}}+\int_{|x+i|> 2|\zeta+i|,\,\, x\in \mathbb{R}}\right)|\widetilde{K}_{n,n}(x, \zeta)|^{p}dx=:I+II.
\end{align}

It is easy to know that
\begin{align}
I&\leq C_{p}\Big\{\int_{|x+i|\leq 2|\zeta+i|,\,\, x\in \mathbb{R}}|K_{n,n}(x-\zeta)|^{p}dx\\
&\,\,\,\,\,\,\,\,+\left(\int_{|x+i|\leq |\zeta+i|,\,\, x\in \mathbb{R}}+\int_{|\zeta+i|<|x+i|\leq 2|\zeta+i|,\,\, x\in \mathbb{R}}\right)|\mathrm{S. P.} [K_{n,n}](x, \zeta)|^{p}dx\Big\}\nonumber\\
&=I_{1}+I_{2}+I_{3},\nonumber
\end{align}
where $C_{p}$ is a constant depending only on $p$.

For $I_{1}$, since $|x-\zeta|\leq 3|\zeta+i|$ when $|x+i|\leq 2|\zeta+i|$ and $x\in \mathbb{R}$, we have
\begin{equation}
|K_{n,n}(x-\zeta)|\leq C_{n}|x-\zeta|^{2(n-1)}\left(1+\big|\log|x-\zeta|\big|\right).
\end{equation}
So
\begin{align}
I_{1}&\leq C(n)|\zeta+i|^{2(n-1)p}\displaystyle\int_{|x+i|\leq 2|\zeta+i|,\,\, x\in \mathbb{R}}\left(1+\big|\log|x-\zeta|\big|\right)^{p}dx\\
&\leq C(n, p)|\zeta+i|^{2(n-1)p}\displaystyle\int_{|x+i|\leq 2|\zeta+i|,\,\, x\in \mathbb{R}}\left(1+\big|\log|x-\zeta|\big|^{p}\right)dx\nonumber\\
&\leq C(n, p)|\zeta+i|^{2(n-1)p+1}\left(1+\big|\log|\zeta+i|\big|^{p}\right)\nonumber\\
&\leq C(n, p)|\zeta+i|^{2(n-1)p+1}\left(1+\big|\log|\zeta+i|\big|\right)^{p},\nonumber
\end{align}
where $C(\cdots)$ are some constants depending only on the quantities in the parentheses (The same conventions are also applied in the above (2.15) and what follows).

For $I_{2}$, in this case, since $|x+i|\leq |\zeta+i|$, by (2.12), then
\begin{equation}
|\mathrm{S. P.} [K_{n,n}](x, \zeta)|\leq C_{n}|\zeta+i|^{2(n-1)}\left(1+\big|\log|\zeta+i|\big|\right).\end{equation}
Therefore,
\begin{align}
I_{2}&\leq C(n)|\zeta+i|^{2(n-1)p}\left(1+\big|\log|\zeta+i|\big|\right)^{p}\displaystyle\int_{|x+i|\leq 2|\zeta+i|,\,\, x\in \mathbb{R}}1dx\\
&\leq C(n)|\zeta+i|^{2(n-1)p+1}\left(1+\big|\log|\zeta+i|\big|\right)^{p}.\nonumber
\end{align}

For $I_{3}$, similar as the last case, we have that
\begin{equation}
|\mathrm{S. P.} [K_{n,n}](x, \zeta)|\leq C_{n}|\zeta+i|^{2(n-1)}\left(1+\big|\log|x+i|\big|\right),\end{equation}
and further
\begin{align}
I_{3}&\leq C(n)|\zeta+i|^{2(n-1)p}\displaystyle\int_{|\zeta+i|<|x+i|\leq 2|\zeta+i|,\,\, x\in \mathbb{R}}\left(1+\big|\log|x+i|\big|\right)^{p}dx\\
&\leq C(n)|\zeta+i|^{2(n-1)p}\displaystyle\int_{|x+i|\leq 2|\zeta+i|,\,\, x\in \mathbb{R}}\left(1+\big|\log|x+i|\big|\right)^{p}dx\nonumber\\
&\leq C(n,p)|\zeta+i|^{2(n-1)p+1}\left(1+\big|\log|\zeta+i|\big|\right)^{p}.\nonumber
\end{align}

From (2.29), (2.31), (2.33) and (2.35), we get

\begin{equation}
I\leq C(n,p)|\zeta+i|^{2(n-1)p+1}\left(1+\big|\log|\zeta+i|\big|\right)^{p}.
\end{equation}

Next turn to $II$. By Definitions 2.8 and 2.9,  in the case of $|x+i|>2|\zeta+i|$ (so $|\frac{\zeta+i}{x+i}|\in(0, \frac{1}{2})$),
\begin{equation}
\widetilde{K}_{n,n}(x, \zeta)=\mathrm{I. P.}[K_{n,n}](x, \zeta)=C(n)\left[c_{n}(x, \zeta)+\widetilde{c}_{n}(x, \zeta)\log|x+i|\right]\frac{1}{|x+i|^{3}},
\end{equation}
 where $C(n)$ is a constant depending only on $n$,
 \begin{align}
 c_{n}(x, \zeta)&=|\zeta+i|^{2n+1}\Big\{\frac{d^{2n+1}}{dr^{2n+1}}\Big[(1-2r\cos(\theta-\vartheta)+r^{2})^{n-1}\\
 &\left[\log(1-2r\cos(\theta-\vartheta)+r^{2})-2\sum_{k=1}^{n-1}\frac{1}{k}\right]\Big]\Big\}_{r=\rho}\nonumber
 \end{align}
and \begin{equation}
 \widetilde{c}_{n}(x, \zeta)=|\zeta+i|^{2n+1}\left\{\frac{d^{2n+1}}{dr^{2n+1}}[(1-2r\cos(\theta-\vartheta)+r^{2})^{n-1}]\right\}_{r=\varrho}
 \end{equation}
 in which $0<\rho, \,\varrho<\frac{|\zeta+i|}{|x+i|}<\frac{1}{2}$. Since $\rho, \,\varrho\in (0, \frac{1}{2})$, then $1-2r\cos(\theta-\vartheta)+r^{2}\in \left(\frac{1}{4}, \frac{9}{4}\right)$. Thus there exists a positive constant $M=M(n)$ depending only on $n$, such that
 \begin{equation}
 |c(x, \zeta)|\leq M(n)|\zeta+i|^{2n+1}\,\,\,\mathrm{and}\,\,\, |\widetilde{c}(x, \zeta)|\leq M(n)|\zeta+i|^{2n+1}.
 \end{equation}

 So
 \begin{align}
 II&\leq C(n)|\zeta+i|^{(2n-2)p+2}\int_{|x+i|> 2|\zeta+i|,\,\, x\in \mathbb{R}}\frac{\left(1+\big|\log|x+i|\big|\right)^{p}}{|x+i|^{2}}dx\\
 &\leq C(n,p)|\zeta+i|^{(2n-2)p+2}\Big[\big(2|\zeta+i|\big)^{-1}\nonumber\\
 &\,\,\,\,\,\,\,\,\,+\int_{|x+i|> 2|\zeta+i|,\,\, x\in \mathbb{R}}\frac{\big||x+i|^{-\frac{1}{2p}}\log|x+i|\big|^{p}}{|x+i|^{\frac{3}{2}}}dx\Big]\nonumber\\
 &\leq C(n,p)|\zeta+i|^{(2n-2)p+2}\Big[\big(2|\zeta+i|\big)^{-1}\nonumber\\
 &\,\,\,\,\,\,\,\,\,+\left(\frac{2p}{e}\right)^{p}\int_{|x+i|> 2|\zeta+i|,\,\, x\in \mathbb{R}}\frac{1}{|x+i|^{\frac{3}{2}}}dx\Big]\nonumber\\
 &\leq C(n,p)|\zeta+i|^{(2n-2)p+2}\Big[\big(2|\zeta+i|\big)^{-1}+\left(\frac{2}{\epsilon e}\right)^{p}\frac{\sqrt{2}}{\sqrt{|\zeta+i|}}\Big]\nonumber\\
 &\leq C(n,p)|\zeta+i|^{(2n-2)p+\frac{3}{2}},\nonumber
 \end{align}
 where the fact $|\zeta+i|\geq 1$ is used. Using such fact once again, (2.27) follows easily from (2.36) and (2.41).

By (2.27), using Minkowski's inequality for integrals and H\"older inequality,
\begin{align}
\|\widetilde{T}_{n, n, \mathbf{H}}f\|_{L^{p}(\mathbb{R})}&\leq \int\int_{\mathbf{H}}\left[\int_{\mathbb{R}}|\widetilde{K}_{n,n}(x, \zeta)|^{p}dx\right]^{\frac{1}{p}}f(\zeta)d\xi d\eta\\
&\leq \left(\int\int_{\mathbf{H}}\left[\int_{\mathbb{R}}|\widetilde{K}_{n,n}(x, \zeta)|^{p}dx\right]^{\frac{q}{p}}w_{1}^{-\frac{q}{p}}(\zeta)d\xi d\eta\right)^{\frac{1}{q}}\times\|f\|_{L_{w_{1}}^{p}(\mathbf{H})}\nonumber\\
&\leq C(n,p)\left(\int\int_{\mathbf{H}}\left[[|\zeta+i|^{2n-2}(1+\big|\log|\zeta+i|\big|)]^{p}|\zeta+i|^{\frac{3}{2}}w_{1}^{-1}(\zeta)\right]^{\frac{1}{p-1}}d\xi d\eta\right)^{\frac{p-1}{p}}\nonumber\\
&\,\,\,\,\,\,\times\|f\|_{L_{w_{1}}^{p}(\mathbf{H})}\nonumber\\
&\leq C(n,p)\|f\|_{L_{w_{1}}^{p}(\mathbf{H})}\nonumber
\end{align}
since $w_{1}\in W^{p, n, \frac{3}{2}}(\mathbf{H})$, where $\frac{1}{p}+\frac{1}{q}=1$. Thus we obtain the desired (2.25) and the theorem completes.
\end{proof}

\begin{thm}[Differentiability of $\widetilde{T}_{n,n, \mathbf{H}}f$] Let $\widetilde{T}_{n,n, \mathbf{H}}$ be as before, then
\begin{equation}
\Delta \widetilde{T}_{n,n, \mathbf{H}}f(z)=4\widetilde{T}_{n-1,n-1, \mathbf{H}}f(z),\,\,\,z\in \mathbf{H}
\end{equation}
for any $n\geq 2$ and $f\in L^{p}(\mathbf{H})$ with $p\geq1$ in the sense of classical derivatives. Moreover,
\begin{equation}
\Delta \widetilde{T}_{1,1, \mathbf{H}}f(z)=4f(z),\,\,\,z\in \mathbf{H}
\end{equation}
for any $f\in L^{p}(\mathbf{H})$ with $p>1$ in the sense of Sobolev derivative.
\end{thm}

\begin{proof}
Invoking Theorems 2.11 and 2.12, (2.43) can be verified by a similar argument as to Theorem 3.3 in \cite{dqw3}. Here the details are completely omitted.

To prove (2.44), we first note that
\begin{align}
\log|z-\zeta|^{2}&=\log\left(1-\frac{z+i}{\zeta+i}\right)+\log\left(1-\frac{\overline{z}-i}{\overline{\zeta}-i}\right)+\log|1+\zeta|^{2}\\
&=-\sum_{l=1}^{\infty}\frac{1}{l}\left[\left(\frac{z+i}{\zeta+i}\right)^{l}+\left(\frac{\overline{z}-i}{\overline{\zeta}-i}\right)^{l}\right]+\log|1+\zeta|^{2}\nonumber
\end{align}
when $|\zeta+i|>|z+i|$, where the analytic branch of $\log z$ is that one cutting along the positive real axis with $\log1=0$. So by Definition 2.8,
\begin{equation}
\mathrm{S. P.} [K_{1,1}](z,\zeta)=\begin{cases}0,\,\,\,\,\,as \,\,|z+i|=|\zeta+i|\,\,and\,\,z\neq\zeta;\vspace{4mm}\\
-\frac{1}{\pi}\left\{\sum_{l=1}^{2}\frac{1}{l}\left[\left(\left(\frac{z+i}{\zeta+i}\right)_{<}\right)^{l}
+\left(\left(\frac{\overline{z}-i}{\overline{\zeta}-i}\right)_{<}\right)^{l}\right]-\log\max\left(|1+z|^{2}, |1+\zeta|^{2}\right)\right\}, \\
\hspace{62mm}as\,\,\,|z+i|\neq|\zeta+i|,
\end{cases}
\end{equation}
where $(z)_{<}=\min\{z, z^{-1}\}$ with $z\neq 0$.
Thus
\begin{equation}
\Delta_{z}\mathrm{S. P.} [K_{1,1}](z,\zeta)=\Delta_{\zeta}\mathrm{S. P.} [K_{1,1}](z,\zeta)=0
\end{equation}
for any $z\neq\zeta$ in the sense of classical derivative. Therefore, in the sense of Sobolev derivatives,
\begin{align}
\Delta_{z}\widetilde{K}_{1,1}(z, \zeta)&=\Delta_{z}K_{1,1}(z, \zeta)+\Delta_{z}\mathrm{S. P.} [K_{1,1}](z,\zeta)\\
&=\Delta_{z}K_{1,1}(z, \zeta)\,\,\,\,\big(\!\!=4\delta_{z}(\zeta)\big)\nonumber\\
&=\Delta_{\zeta}K_{1,1}(z, \zeta)\,\,\,\,\big(\!\!=4\delta_{\zeta}(z)\big)\nonumber\\
&=\Delta_{\zeta}K_{1,1}(z, \zeta)+\Delta_{\zeta}\mathrm{S. P.} [K_{1,1}](z,\zeta)\nonumber\\
&=\Delta_{\zeta}\widetilde{K}_{1,1}(z, \zeta)\nonumber,
\end{align}
where $\delta_{z}$ denotes the Dirac function on $\mathbf{H}$ taking the unit mass to the point $z\in \mathbf{H}$.

Since $f\in L^{p}(\mathbf{H})$ with $p>1$, by Lebesgue's dominated convergence theorem, it is easy to know that $\widetilde{T}_{1,1}f\in C_{\mathrm{loc}}(\mathbf{H})$. Then for any $\varphi\in C^{\infty}_{0}(\mathbf{H})$, we have
\begin{align}
\int\int_{\mathbf{H}}\widetilde{T}_{1,1}f(z)\Delta_{z}\varphi(z)dxdy&=\int\int_{\mathbf{H}}\left(\int\int_{\mathbf{H}}\widetilde{K}_{1,1}(z, \zeta)f(\zeta)d\xi d\eta\right)\Delta_{z}\varphi(z)dxdy\\
&=\int\int_{\mathbf{H}}\left(\int\int_{\mathbf{H}}\widetilde{K}_{1,1}(z, \zeta)\Delta_{z}\varphi(z)dxdy\right)f(\zeta)d\xi d\eta\nonumber\\
&=\int\int_{\mathbf{H}}\left(\int\int_{\mathbf{H}}\Delta_{z}\widetilde{K}_{1,1}(z, \zeta)\varphi(z)dxdy\right)f(\zeta)d\xi d\eta\nonumber\\
&=\int\int_{\mathbf{H}}\left(\int\int_{\mathbf{H}}4\delta_{\zeta}(z)\varphi(z)dxdy\right)f(\zeta)d\xi d\eta\nonumber\\
&=\int\int_{\mathbf{H}}4f(\zeta)\varphi(\zeta)d\xi d\eta,\nonumber
\end{align}
where the second equality is justified by Fubini's theorem and Theorem 2.16 since $\Delta_{z}\varphi\in C^{\infty}_{0}(\mathbf{H})\in L_{w}^{p}(\mathbf{H})$ for any weight $w\in W^{q,1,\frac{5}{2}}(\mathbf{H})$ with $q>1$; the third equality is immediately from the integral by parts; and the penultimate equality is obtained by (2.48). However, (2.49) is just the desired (2.44) in the sense of Sobolev derivative.
\end{proof}

\begin{thm}[Nontangential boundary value of $\widetilde{T}_{n,n, \mathbf{H}}f$] Let $\widetilde{T}_{n,n, \mathbf{H}}$ and $f$ be as in the above theorem, then
\begin{equation}
\lim_{\substack{z\rightarrow s\\ z\in\Gamma_{\alpha}(s), s\in \mathbb{R}}}\widetilde{T}_{n,n, \mathbf{H}}f(z)=\widetilde{T}_{n,n, \mathbf{H}}f(s)
\end{equation}
for any $n\in \mathbb{N}$.
\end{thm}

\begin{proof}
In virtue of Theorem 2.11, it follows from a similar argument as to the verification for the property 5 of the higher order Poisson kernels defined in \cite{dqw1} in terms of the estimate techniques of integral splitting. Such techniques have been used in the proof of the above Theorem 2.16. The details can be found in the argument of Theorem 2.7 in \cite{dqw3}, and omitted again here.
\end{proof}

\section{Uniqueness of solution for homogeneous PHD problem in the upper half-plane }

In order to get the uniqueness of the solution of the inhomogeneous PHD problem under question, in this section, we will give a certain estimate under which the solution of the corresponding homogeneous PHD problem obtained in \cite{dqw1} is unique. To do so, we need exploit the following standard theorems in classical harmonic analysis.

\begin{lem}[\cite{st}] Let $f\in L^{p}(\mathbb{R})$, $1\leq p\leq \infty$, and $u(z)$ be the Poisson integral of $f$ (in our notations, $u(z)=\frac{1}{\pi}\int_{-\infty}^{+\infty}G_{1}(z,t)f(t)dt$, $z\in \mathbf{H}$),  then
\begin{equation}
\mathrm{M}[u](x_{0})=\sup_{z\in\Gamma_{\alpha}(x_{0})}|u(z)|\leq C_{\alpha} \mathfrak{M}f(x_{0}),
\end{equation}
where $\Gamma_{\alpha}(x_{0})$ is the cone in $\mathbf{H}$ with the vertex at $(x_{0}, 0)$ and the aperture $\alpha$, $x_{0}\in \mathbb{R}$, $\alpha>0$; $C_{\alpha}$ is a positive constant depending only on $\alpha$, $\mathrm{M}[u]$ is the non-tangential maximal function,  and $\mathfrak{M}f$ is the standard Hardy-Littlewood maximal function defined by
\begin{equation}
\mathfrak{M}f(x_{0})=\sup_{\rho>0}\frac{1}{2\rho}\int_{x_{0}-\rho}^{x_{0}+\rho}|f(x)|dx.
\end{equation}
\end{lem}

\begin{lem}[Hardy-Littlewood maximal theorem, see \cite{gar}] Let $f\in L^{p}(\mathbb{R})$, $1\leq p\leq \infty$, then $\mathfrak{M}f$ is finite almost everywhere on $\mathbb{R}$ . Moreover,
\begin{itemize}
  \item [(1)] If $f\in L^{1}(\mathbb{R})$, then $\mathfrak{M}f$ is in weak-$L^{1}(\mathbb{R})$ (which is usually denoted by $L^{1,\infty}(\mathbb{R})$), more precisely,
  \begin{equation}
  |\{x\in \mathbb{R}: \mathfrak{M}f(x)>\lambda\}|\leq \frac{2}{\lambda}\|f\|_{L^{1}(\mathbb{R})};
  \end{equation}
  \item [(2)] If $f\in L^{p}(\mathbb{R})$, $1<p\leq\infty$, then
  \begin{equation}
  \|\mathfrak{M}f\|_{L^{p}(\mathbb{R})}\leq A_{p}\|f\|_{L^{p}(\mathbb{R})},
  \end{equation}
  where $|E|$ denotes Lebesque's measure of the set $E$, $A_{p}$ is a constant depending only on $p$, $\mathfrak{M}f$ and $\alpha$ be as in the above lemma.
\end{itemize}
\end{lem}

By Lemmas 3.1-3.2, immediately, we have
\begin{cor}
Let $\mathrm{M}$, $\alpha$ and $u$ be as above, then
\begin{equation}
\|\mathrm{M}[u]\|_{L^{p}(\mathbb{R})}\leq C_{p,\alpha}\|f\|_{L^{p}(\mathbb{R})}
\end{equation}
for any $f\in L^{p}(\mathbb{R})$ with $1< p\leq\infty$, where $C_{p,\alpha}$ is a constant depending only on $p,\alpha$. Moreover,
\begin{equation}
  |\{x\in \mathbb{R}: \mathrm{M}[u](x)>\lambda\}|\leq \frac{2C_{\alpha}}{\lambda}\|f\|_{L^{1}(\mathbb{R})}
  \end{equation}
  for any $f\in L^{1}(\mathbb{R})$, and for any $f\in L^{p}(\mathbb{R})$, $1\leq p\leq\infty$, $\mathrm{M}[u]$ is finite almost everywhere on $\mathbb{R}$, $C_{\alpha}$ is the same as in Lemma 3.1.
\end{cor}

Now we can give the existence and uniqueness of solution to the corresponding homogeneous PHD problem associated with the inhomogeneous PHD problem (1.1).

\begin{thm} Let $\{G_n(z,t)\}^{\infty}_{n=1}$ be the sequence of higher order Poisson kernels defined on $\mathbf{H}\times \mathbb{R}$, which are defined in Definition 2.1 and explicitly expressed in Lemma 2.5, then for any $n\geq1$, the following homogeneous PHD problem
\begin{equation}
\begin{cases}
\Delta^{n}u=0\,\,\,in\,\,\mathbf{H},\\
\Delta^{j}u=f_{j}\,\,\,on\,\,\mathbb{R},
\end{cases}
\end{equation}
where $f_{j}\in L^{p}(\mathbb{R})$, $0\leq j\leq n-1$, $1<p<\infty$, and the boundary data are non-tangential, is solvable and a solution is
\begin{equation}u(z)=\sum\limits^{n}_{j=1}\frac{4^{1-j}}{\pi}\int^{+\infty}_{-\infty}G_{j}(z,t)f_{j-1}(t)dt\triangleq\sum\limits^{n}_{j=1}M_{j}f_{j-1}(z)\end{equation}
which satisfying that
\begin{equation}
\|\mathrm{M}[u-\sum\limits^{n}_{j=2}M_{j}f_{j-1}]\|_{L^{p}(\mathbb{R})}\leq C_{p,\alpha}\|f_{0}\|_{L^{p}(\mathbb{R})},
\end{equation}
where $C_{p,\alpha}$ is a constant depending only on $p$ and $\alpha$, the operator $M_{j}$ is called to be the $j$-th order Poisson integral operator defined by
\begin{equation}
M_{j}f(z)=\frac{4^{1-j}}{\pi}\int^{+\infty}_{-\infty}G_{j}(z,t)f(t)dt
\end{equation}
for some appropriate $f$. Moreover, the solution (3.8) is unique under the estimate of type (3.9).
\end{thm}

\begin{proof}
The existence of solution of the form as (3.8) was the main result in \cite{dqw1}. To the solution (3.8), the estimate (3.9) readily follows from Corollary 3.3.
Now turn to the uniqueness. Suppose that $v$ is another solution of the homogeneous PHD problem (3.7) under the estimate of type (3.9), then $u-v$ is a solution of the following homogeneous PHD problem with vanishing boundary data as follows
\begin{equation}
\begin{cases}
\Delta^{n}u=0\,\,\,in\,\,\mathbf{H},\\
\Delta^{j}u=0\,\,\,on\,\,\mathbb{R}.
\end{cases}
\end{equation}
Now the corresponding estimate of type (3.9) associated with $u-v$ is that $\|\mathrm{M}[u]\|_{L^{p}(\mathbb{R})}\leq 0$. So by the definition of $\mathrm{M}[u]$, the desired conclusion,  that $u=0$ a.e. on $\mathbf{H}$, follows immediately.
\end{proof}

\section{Inhomogeneous PHD problem in the upper half-plane}

With the above preliminaries, we can conclude with the main result of the present paper in this section. That is the following

\begin{thm}
Let $\widetilde{T}_{j, \mathbf{H}}$, $M_{j}$, $\mathrm{M}$ and $\alpha$ be as before, $1\leq j\leq n$, then for any $f_{l}\in L^{p}(\mathbb{R})$ and $g\in L_{w}^{p}(\mathbf{H})$ with $w\in W^{p,n,\frac{3}{2}}(\mathbf{H})$, $0\leq l\leq n-1$, the inhomogeneous PHD problem (1.1) is solvable and a solution is given by
\begin{equation}u(z)=\frac{1}{4^{n}}\widetilde{T}_{n,\mathbf{H}}g(z)+\sum\limits_{j=1}^{n}M_{j}\left(f_{j-1}-\frac{1}{4^{n+1-j}}\widetilde{T}_{n+1-j, \mathbf{H}}g\right)(z) \end{equation}
which satisfying the following estimate
\begin{equation}\|{\rm{M}}[u-\frac{1}{4^{n}}\widetilde{T}_{n,\mathbf{H}}g-\sum_{j=2}^{n}M_{j}(f_{j-1}-\frac{1}{4^{n+1-j}}\widetilde{T}_{n+1-j, \mathbf{H}}g)]\|_{L^{p}(\mathbb{R})}\leq C\left(\|f_{0}\|_{L^p(\mathbb{R})}+\|g\|_{L_{w}^{p}(\mathbf{H})}\right),\end{equation} where $C$ is a constant depending only on $n,p$ and $\alpha$.  Moreover, the solution (4.1) is unique under such typed estimate (4.2).\end{thm}

\begin{proof}
Since $g\in L^{p}_{w}(\mathbf{H})\subset L^{p}(\mathbf{H})$ with $w\in W^{p,n,\frac{3}{2}}(\mathbf{H})$ (see Remark 2.17), let $h=\widetilde{T}_{n, \mathbf{H}}g$, then by Theorem 2.18, $\Delta^{j}h=4^{j}\widetilde{T}_{n-j, \mathbf{H}}g$ (in classical sense) and $\Delta^{n}h=4^{n}g$ (in Sobolev sense) in $\mathbf{H}$, $1\leq j\leq n-1$. Set $v=u-\frac{1}{4^{n}}h$, by these facts and Theorem 2.19, so $v$ satisfies the following homogeneous PHD problem
\begin{equation}
\begin{cases}
\Delta^{n}v=0 \ \ in \ \textbf{H},\\
\Delta^{j}v=f_{j}-\frac{1}{4^{n-j}}\widetilde{T}_{n-j, \mathbf{H}}g \ on \ \mathbb{R},\\
\end{cases}
\end{equation}
where $0\leq j\leq n-1$. Noting Remark 2.15, by invoking Theorem 2.16, we know that
\begin{equation}
\|\widetilde{T}_{n-j, \mathbf{H}}g\|_{L^{p}(\mathbb{R})}\leq C\|g\|_{L_{w}^{p}(\mathbf{H})}
\end{equation}
for any $0\leq j\leq n-1$ and $g\in L^{p}_{w}(\mathbf{H})$ with $w\in W^{p,n,\frac{3}{2}}(\mathbf{H})$, where $C$ are some constants depending only on $n,j$ and $p$. So by Theorem 3.4, a solution of (4.3) can be given by
\begin{align}v(z)&=\sum\limits^{n}_{j=1}\frac{4^{1-j}}{\pi}\int^{+\infty}_{-\infty}G_{j}(z,t)\left[f_{j-1}(t)-\frac{1}{4^{n+1-j}}\widetilde{T}_{n+1-j, \mathbf{H}}g\right]dt\\
&=\sum\limits^{n}_{j=1}M_{j}\left[f_{j-1}(t)-\frac{1}{4^{n+1-j}}\widetilde{T}_{n+1-j, \mathbf{H}}g\right](z),\nonumber\end{align}
and it is unique under the following estimate
\begin{equation}
\|\mathrm{M}[v-\sum\limits^{n}_{j=2}M_{j}(f_{j-1}-\frac{1}{4^{n+1-j}}\widetilde{T}_{n+1-j, \mathbf{H}}g)]\|_{L^{p}(\mathbb{R})}\leq C_{p,\alpha}\|f_{0}-\frac{1}{4^{n}}\widetilde{T}_{n, \mathbf{H}}g\|_{L^{p}(\mathbb{R})}.
\end{equation}
Therefore, by (4.4), the unique solution $u$ of the problem (1.1) under the estimate of type (4.2) is given by (4.1). Thus the theorem completes.
\end{proof}

\begin{rem}[The estimate in the case of $w\in W^{p,n,\frac{5}{2}}(\mathbf{H})$] By (2.26) in Theorem 2.16, the solution (4.1) also satisfies the following estimate
 \begin{equation}\|u-\sum_{j=1}^{n}M_{j}(f_{j-1}-\frac{1}{4^{n+1-j}}\widetilde{T}_{n+1-j, \mathbf{H}}g)]\|_{L^{p}(\mathbf{H})}\leq C\|g\|_{L_{w}^{p}(\mathbf{H})},\end{equation}
where $C$ is a constant depending only on $n$ and $p$. It is easy to obtain that the solution (4.1) is also unique under the estimate of type (4.7).
\end{rem}

\begin{rem}[Green functions for polyharmonic operators on $\mathbf{H}$] Rewrite (4.1) as
\begin{equation}u(z)=\Big[\frac{1}{4^{n}}\widetilde{T}_{n,\mathbf{H}}-\frac{1}{4^{n+1-j}}\sum\limits_{j=1}^{n}M_{j}\widetilde{T}_{n+1-j, \mathbf{H}}\Big]g(z)+\sum\limits_{j=1}^{n}M_{j}f_{j-1}(z), \end{equation}
so from the first term in RHS of (4.8), we obtain a Green function for $\Delta^{n}$ on $\mathbf{H}$ by
\begin{align}
\mathcal{G}_{n}(z,\zeta)&=\frac{1}{4^{n}}\Big[\widetilde{K}_{n,n}(z, \zeta)-\frac{1}{\pi}\sum\limits^{n}_{j=1}\widetilde{K}_{n+1-j,n=+1-j}\int_{\mathbb{R}}G_{j}(z,t)\widetilde{K}_{n+1-j,n+1-j}(t, \zeta)dt\Big]\\
&=\frac{1}{4^{n}}\Big[\widetilde{K}_{n,n}(z, \zeta)-\frac{4^{j-1}}{\pi}\sum\limits^{n}_{j=1}M_{j}\widetilde{K}_{n+1-j,n+1-j}(z, \zeta)\Big].\nonumber\end{align} By noting the estimates of $G_{n}$ and $\widetilde{K}_{n,n}$, stated respectively in Definition 2.1 (i.e., the property (2)) and Theorem 2.11, we know that the integrals in this formula are absolutely convergent for arbitrary $z,\zeta\in \mathbf{H}$ and $z\neq\zeta$. It is easy to get that $\mathcal{G}_{n}(z,\zeta)= \mathcal{G}_{n}(\zeta,z)$ for $z,\zeta\in \mathbf{H}$ with $z\neq \zeta$, and $\mathcal{G}_{n}(\cdot,\zeta)=0$ on $\mathbb{R}$ for any $\zeta\in \mathbf{H}$.
\end{rem}


\end{document}